\documentclass[12pt]{amsart}
\usepackage{amscd,amssymb}
\usepackage{pgfplots}
\pgfplotsset{compat=1.15}
\usepackage{mathrsfs}
\usetikzlibrary{arrows}
\usepackage{url}
\usepackage[arrow,matrix]{xy}
\usepackage{amscd,amssymb,latexsym,subfigure,hyperref}
\theoremstyle{plain}
\newtheorem{thm}[subsection]{Theorem}

\newtheorem{lem}[subsection]{Lemma}
\newtheorem{prop}[subsection]{Proposition}
\newtheorem{cor}[subsection]{Corollary}
\theoremstyle{definition}
\newtheorem{rk}[subsection]{Remark}
\newtheorem{definition}[subsection]{Definition}
\newtheorem{ex}[subsection]{Example}
\newtheorem{conj}[subsection]{Conjecture}

\numberwithin{equation}{section}
\newcommand{\A}{{\mathcal A}}

\newcommand{\CC}{{\mathcal C}}

\newcommand{\al}{{\alpha}}
\newcommand{\be}{{\beta}}

\newcommand{\C}{\mathbb{C}}
\newcommand{\PP}{\mathbb{P}}

\begin{document}
\title [On conic-line arrangements in the plane]
{On conic-line arrangements with nodes, tacnodes, and ordinary triple points}

\author[Alexandru Dimca]{Alexandru Dimca}
\address{Universit\'e C\^ ote d'Azur, CNRS, LJAD, France and Simion Stoilow Institute of Mathematics,
P.O. Box 1-764, RO-014700 Bucharest, Romania}
\email{dimca@unice.fr}

\author[Piotr Pokora]{Piotr Pokora$^{\ast}$}.
\address{Department of Mathematics,
Pedagogical University of Krakow,
Podchor\c a\.zych 2,
PL-30-084 Krak\'ow, Poland}
\email{piotr.pokora@up.krakow.pl}

\subjclass[2010]{Primary 14N20; Secondary  14C20, 32S22}

\keywords{conic-line arrangements, nodes, tacnodes, freeness, nearly freeness}

\begin{abstract} In the present paper, we study conic-line arrangements having nodes, tacnodes, and ordinary triple points as singularities. We provide combinatorial constraints on such arrangements and we give the complete classification of free arrangements in this class. 
\end{abstract}
 
\maketitle

\section{Introduction}
In the present paper we study  a class of conic-line arrangements in the complex projective plane $\mathbb{P}^{2}_{\mathbb{C}}$, with special attention to the free arrangements in this class. The theory of free line arrangements is rather rich and we have many results which provide (at least) a partial characterization of the freeness. In that subject, the ultimate goal is to understand whether Terao's Conjecture is true in its whole generality. On the other hand, line arrangements show up naturally in algebraic geometry. For example, Hirzebruch's inequality appreciated very much in combinatorics is motivated by many extreme problems in algebraic geometry and it is obtained with its methods. Based on that, it seems to be quite natural to extend this set-up to higher degree curves. From our perspective it seems very natural to start working on arrangement consisting of rational curves in the plane. Here we study arrangements of smooth conics and lines in the plane. The first main motivation is that conic-line arrangements admit non-ordinary singularities, so we can study arrangements having, for instance, tacnodes as singularities. 

In general, singularities of such arrangements are in general not quasi-homogeneous which makes their study quite complicated.
By \cite[Example 4.1]{SchenckToh}, we know that Terao's conjecture does not hold in general for such arrangements. Let us recall the following counterexample from the aforementioned paper.
\begin{ex}
Consider the following conic-line arrangement
$$\mathcal{CL}_{1} \, : \, xy\cdot(y^{2} + xz)\cdot(y^{2}+x^{2} + 2xz)=0.$$
The intersection point $P = (0:0:1)$ has multiplicity $4$ and it is quasi-homogeneous (although it is not \textbf{ordinary}). One can show that $\mathcal{CL}_{1}$ is free with the exponents $(2,3)$. If we perturb a bit line $y=0$, taking for instance $x-13y = 0$, we obtain a new conic-line arrangement 
$$\mathcal{CL}_{2} \, : \, x\cdot(x-13y)\cdot(y^{2} + xz)\cdot(y^{2}+x^{2} + 2xz)=0.$$
It is worth emphasizing that the Levi graphs for $\mathcal{CL}_{1}$ and $\mathcal{CL}_{1}$ are the same, so they have the same (strong) combinatorics. However, this new arrangement, the intersection point $P=(0:0:1)$ has multiplicity $4$, but it is not longer quasi-homogeneous, and $\mathcal{CL}_{2}$ is not free. In fact, the arrangement $\mathcal{CL}_{2}$ is nearly-free.
\cite{DimcaSernesi} and \cite{STY}.
\end{ex}
In the present paper we focus on conic-line arrangements in the plane such that their singularities are \textbf{nodes, tacnodes, and ordinary triple points}. This assumption is mostly related with our main scope, to verify Terao's Conjecture for as large as possible class of conic-line arrangements.
 One of our results, Proposition \ref{prop:bound9}, tells us that if $C$ is a \textbf{free} reduced plane curve of degree $m$ having only nodes, tacnodes, and ordinary triple points, then $m\leq 9$. Based on that combinatorial restriction, we can perform a detailed search in order to find conic-line arrangements with nodes, tacnodes, and ordinary triple points that are free. Our main result, Corollary \ref{cor5}, tells us that the so-called \textbf{Numerical Terao's Conjecture holds for our class of conic-line arrangements}. As it was mentioned at the beginning of this section, curve arrangements attract researchers working both in algebraic geometry and combinatorics. Due to these reasons, we provide combinatorial constraints on the weak combinatorics of conic-line arrangements with nodes, tacnodes, and ordinary triple points. We deliver a Hirzebruch-type inequality for such arrangements, see Theorem \ref{hirz}, and this theorem is in the spirit of results presented in \cite{PokSz}. Then, using the properties of spectra of singularities as in the seminal paper by Varchenko \cite{V}, we provide bounds on the number of tacnodes and ordinary triple points, see Theorem \ref{prop2}.

\section{Hirzebruch-type inequality for conic-line arrangements}
We start with presenting our set-up. Let $\mathcal{CL} = \{\ell_{1}, ..., \ell_{d}, C_{1}, ..., C_{k} \} \subset \mathbb{P}^{2}_{\mathbb{C}}$ be an arrangement consisting of $d$ lines and $k$ smooth conics. We assume that our conic-line arrangements have $n_{2}$ nodes, $t$ tacnodes, and $n_{3}$ ordinary triple points. We have the following combinatorial count
\begin{equation}
\label{comb:naiv}
4 \binom{k}{2} + 2kd + \binom{d}{2} = n_{2} + 2t + 3n_{3}.
\end{equation}
\begin{proof}
Observe that the left-hand side is the number of pairwise intersections of curves contained in $\mathcal{CL}$. The right-hand side, according to B\'ezout's theorem, is based on the intersection indices. If $p$ is a node, then the intersection index of curves meeting at that point is equal to $1$. If $p$ is a tacnode, then the intersection index of curves meeting at that point is equal to $2$. Finally, if $p$ is an ordinary triple points, that the intersection index of curves meeting at that point is equal to $3$. This completes our justification.
\end{proof}
The first result of the present paper is the following Hirzebruch-type inequality.
\begin{thm}
\label{hirz}
Let $\mathcal{CL} = \{\ell_{1}, ..., \ell_{d}, C_{1}, ..., C_{k}\} \subset\mathbb{P}^{2}_{\mathbb{C}}$ be an arrangement of $d$ lines and $k$ smooth conics and such that $2k+d \geq 12$. Assume that $\mathcal{CL}$ has only $n_{2}$ nodes, $t$ tacnodes, and $n_{3}$ ordinary triple points. Then
\begin{equation}
    20k +n_{2} + \frac{3}{4}n_{3} \geq d + 4t.
\end{equation}
\end{thm}
We will prove the above theorem using Langer's variation on the Miyaoka-Yau inequality \cite{Langer} which involves the local orbifold Euler numbers $e_{orb}$ of singular points. We recall basics on them in a concise way. Let $(\mathbb{P}^{2}_{\mathbb{C}}, \alpha C)$ be an effective and log-canonical pair, where $C$ is a boundary divisor having only nodes, tacnodes, and ordinary triple points as singularities. Then
\begin{itemize}
    \item if $q$ is a node, then  the local orbifold Euler number is equal to $e_{orb}(p,\mathbb{P}^{2}_{\mathbb{C}}, \alpha C)=(1-\alpha)^2$ provided that $0 \leq \alpha \leq 1$,
    \item if $q$ is a tacnode, then $e_{orb}(q,\mathbb{P}^{2}_{\mathbb{C}}, \alpha C) =(1-2 \alpha)$ provided that $0 \leq \alpha \leq \frac{1}{4}$
    \item if $q$ is an ordinary triple point, then $e_{orb}(q,\mathbb{P}^{2}_{\mathbb{C}}, \alpha C) \leq \bigg(1 - \frac{3\alpha}{2}\bigg)^2$ provided that $0 \leq \alpha \leq \frac{2}{3}$
\end{itemize}
Here we are read to present our proof of Theorem \ref{hirz}.
\begin{proof}
Let $C:= \ell_{1} + ... + \ell_{d} + C_{1} + ... + C_{k}$ be a divisor associated with $\mathcal{CL}$ and such that $m = {\rm deg} \, C = d + 2k \geq 12$ -- we will see in a moment why the last assumption is crucial. First of all, we need to choose $\alpha$ in such a way that $K_{\mathbb{P}^{2}_{\mathbb{C}}} + \alpha C$ is effective and log canonical. In order to obtain the effectivity of the pair, one needs to satisfy the condition that $-3 + \alpha(2k+d) \geq 0$ which implies $\alpha \geq 3/(2k+d)$. On the other hand, our pair is log-canonical if $\alpha \leq {\rm min} \{1, 2/3, 1/4\}$, so $\alpha \leq 1/4$. Due to these two reasons, we get $\alpha \in \bigg[3/(2k+d), 1/4 \bigg]$, and this condition is non-empty provided that $2k+d \geq 12$. From now on we take $\alpha = \frac{1}{4}$, and then apply the following inequality
\begin{equation}
\label{logMY}
\sum_{p \in {\rm Sing}(C)}  3\bigg( \alpha (\mu_{p} - 1) + 1 - e_{orb}(p,\mathbb{P}^{2}_{\mathbb{C}}, \alpha C) \bigg) \leq (3\alpha - \alpha^{2})m^{2} - 3\alpha m,
\end{equation}
where $\mu_{p}$ is the Milnor number of a singular point $p \in {\rm Sing}(C)$.
This gives us
$$3n_{2}\bigg(\frac{1}{4}(1-1)+1 - (1-1/4)^{2}\bigg) + 3t\bigg(\frac{1}{4}(3-1)+ 1 - (1-1/2)\bigg) + 3n_{3}\bigg(\frac{1}{4}(4-1) + 1 -(1 - 3/8)^{2}\bigg) $$
$$ \leq \sum_{p \in {\rm Sing}(C)}  3\bigg( \alpha (\mu_{p} - 1) + 1 - e_{orb}(p,\mathbb{P}^{2}_{\mathbb{C}}, \alpha C) \bigg) .$$
After some simple manipulations, we obtain
$$\frac{21}{16}n_{2} + 3t + \frac{261}{4}n_{3} \leq \frac{11}{16}m^{2} - \frac{3}{4}m.$$
Using (\ref{comb:naiv}), we have $m^2 = 2n_{2} + 4t + 6n_{3} + 4k+d$, and this implies 
$$11m^{2} - 12m = 22n_{2}+44t+66n_{3}+44k+11d - 24k - 12d = 22n_{2}+44t + 66n_{3} + 20k - d.$$
Combining the above computations together, we get
$$21n_{2} + 48t + \frac{261}{4}n_{3} \leq 22n_{2} +44t + 66n_{3} +20k - d,$$
which completes the proof.
\end{proof}

\section{Spectra of singularities and constraints on conic-line arrangements}
Let  $F_{\bullet}: \, H = F^{0} \supset ... \supset F^{p} \supset F^{p+1} \supset ...$ be the Hodge filtration on the (reduced) vanishing cohomology $H = H^{n}(X_{\infty})$ of an isolated singularity $f$, where $X_{\infty}$ denotes the canonical Milnor fibre. The filtration $F_{\bullet}$ is invariant with respect to the action of the semisimple part of the monodromy $T_{s}$. Hence $T_{s}$ acts on ${\rm Gr}_{F}^{p}H = F^{p}/F^{p+1}$ and ${\rm Gr}_{F}^{p}H = \bigoplus_{\lambda} ({\rm Gr}_{F}^{p})_{\lambda}$, where $({\rm Gr}_{F}^{p})_{\lambda} = {\rm Gr}_{F}^{p}H_{\lambda}$ is the eigensubspace corresponding to $\lambda$. Denote by
$$\mu_{p} = {\rm dim} \, {\rm Gr}_{F}^{p},\, \quad \mu_{\lambda}^{p} = {\rm dim} ({\rm Gr}_{F}^{p})_{\lambda}.$$
Then $\sum_{p} \mu^{p} = \mu$ is the Milnor number, $\sum_{\lambda} \mu_{\lambda}^{p} = \mu^{p}$, and $\sum_{p}\mu_{\lambda}^{p} = \mu_{\lambda}$ is the multiplicity of an eigenvalue $\lambda$, where $\mu_{\lambda} = {\rm dim} H_{\lambda}$.  Now to each eigenvalue $\lambda$ one defines 
$$\alpha = -(1/2\pi\iota){\rm log}(\lambda),$$
where $\iota^2=-1$. Since $\lambda$ is a root of the unity, $\alpha$ is a rational number defined modulo an integer. We normalize $\alpha$ according to the level $p$ of $\lambda$ with respect to $F_{\bullet}$ by the condition
$$\alpha = - \frac{1}{2\pi \iota} {\rm log} \, \lambda, \quad n-p-1 < \alpha \leq n-p,$$
where this $\lambda$ comes from the action of $T_s$ on ${\rm Gr}_{F}^{p}H$.
In this way, one obtains an element of the group $\mathbb{Z}^{\mathbb{Q}}$ of the form
$${\rm Sp}(f) = (\alpha_{1}) + ... + (\alpha_{\mu}) = \sum_{\alpha} n_{\alpha}\cdot (\alpha),$$
with $n_{\alpha} = \mu_{\lambda}^{p}$, which is called \textbf{the spectrum of the singularity}. The numbers $\alpha$ are called spectral numbers, and the coefficients $n_{\alpha} $ are the spectral multiplicities. Let us recall some basic properties of spectral numbers for isolated singularities of hypersurfaces.
\begin{enumerate}
\item $\alpha_{j} \in (0,n)$ if $n= \dim X$.
\item The spectrum is an invariant of a singularity.
\item Symmetry: $\alpha_{i} = \alpha_{\mu - i}$.
\item \emph{Thom-Sebastiani Principle}: If $f \in \mathbb{C}\{x_{0}, \ldots x_{m}\}$ and $g \in \mathbb{C}\{y_{0}, \ldots, y_{n}\}$ are two series in separate sets of variables, the expression
$$f \oplus g = f(x_{0}, \ldots x_{m}) + g(y_{0}, \ldots, y_{n}) \in \mathbb{C}\{x_{0}, \ldots x_{m},y_{0}, \ldots, y_{n}\}$$ is called the Thom-Sebastiani sum of $f$ and $g$. Then
$${\rm Sp}(f\oplus g) = \{ \alpha + \beta \, : \, \alpha \in {\rm Sp}(f), \beta\in {\rm Sp}(g)\}.$$
\item ${\rm Sp}(x^{m}) = \{ \frac{1}{m}, \frac{2}{m}, ..., \frac{m-1}{m}\}.$
\end{enumerate}
Using the formulae above, we can compute spectral numbers for nodes $(A_{1})$, tacnodes $(A_{3})$, and ordinary triple points $(D_{4})$.
\begin{enumerate}
\item[$(A_{1})$:] This singular point can be locally described by $x^{2} + y^{2} = 0$, so ${\rm Sp}(x^{2}) = \{\frac{1}{2}\}$, ${\rm Sp}(y^{2}) = \{\frac{1}{2}\}$, and then ${\rm Sp}(A_{1}) = 1\cdot 1$.
\item[$(A_{3})$:] This singular point can be locally described by $y^{2} + x^{4} = 0$, so ${\rm Sp}(y^{2}) = \{\frac{1}{2}\}$, ${\rm Sp}(x^{4}) = \{\frac{1}{4}, \frac{1}{2}, \frac{3}{4}\}$, and we obtain ${\rm Sp}(A_{3}) = 1 \cdot \frac{3}{4} + 1 \cdot 1 + 1 \cdot \frac{5}{4}$.
\item[$(D_{4})$:] This singular point can be locally described by $x^{3} + y^{3} = 0$, so ${\rm Sp}(x^{3}) = {\rm Sp}(y^{3}) = \{\frac{1}{3}, \frac{2}{3}\}$, and ${\rm Sp}(D_{4}) = 1\cdot \frac{2}{3} + 2 \cdot 1 + 1 \cdot \frac{4}{3}$. 
\end{enumerate}
Now we present the main result of this section.
\begin{thm}
\label{prop2}
Let $\mathcal{CL} = \{\ell_{1}, ..., \ell_{d}, C_{1}, ..., C_{k}\} \subset \mathbb{P}^{2}_{\mathbb{C}}$ be an arrangement of $d\geq 0$ lines and $k\geq 0$ smooth conics. Assume that $\mathcal{CL}$ has only $n_2$ nodes, $t$ tacnodes, and $n_3$ ordinary triple points. Let $C = \ell_{1} +...+\ell_{d}+C_{1} + ... + C_{k}$ and write  $m: = {\rm deg} \, C = d+2k$ as $m =3m' +\epsilon$ with $\epsilon \in \{1,2,3\}$.
Then one has
$$t + n_3\leq \binom{m-1}{2} +k  - \frac{m'(5m'-3)}{2}$$
and
$$n_3 \leq (m'+1)(2m'+1).$$
\end{thm}
\begin{proof}
We are going to use the theory of spectra of singularities. Recall that  if $(X,0)$ is the union of $m$ lines passing through the origin of $\C^2$, then the corresponding spectrum is
$${\rm Sp}(X,0)= \sum_{j=1}^{m-1} j \cdot \frac{j+1}{m}  +\sum_{j=2}^{m-1}(m-j) \cdot \frac{m + j-1}{m} .$$
We apply the semicontinuity property of the spectrum in the form presented by Steenbrink in \cite{S} (see also \cite[Theorem 8.9.8]{Ku}) for the semicontinuity domain $B=(\frac{1}{3},\frac{4}{3}]$.
If $L$ is a generic line in $\mathbb{P}^{2}_{\mathbb{C}}$, then the trace of the arrangement  $\mathcal{CL}$ on the complement $\mathbb{C}^2=\mathbb{P}^{2}_{\mathbb{C}} \setminus L$ can be identified with a deformation $X_s$ of a singularity of type $(X,0)$ introduced above. Therefore we get the following equation
\begin{equation}
\label{E1}
{\rm deg}_B \sum_y {\rm Sp}(X_s, y) = n_2 + 3t + 4n_3,
\end{equation}
where the above sum is over singular points $y \in X_s$, and ${\rm deg}_{B} \, \sum_y {\rm Sp}(X_s , y)$ denotes the sum of all spectral multiplicities for spectral numbers that are contained in domain $B$.

On the other hand, the total degree of the spectrum ${\rm Sp}(X,0)$ is equal to the Milnor number $\mu(X,0)=(m-1)^2$. To get the degree for the restriction of the spectrum ${\rm Sp}(X,0)$ to the interval $B=(\frac{1}{3},\frac{4}{3}]$, we have to subtract the sum $S_1$ of the multiplicities of the spectral numbers $\alpha$
such that  $\alpha\leq \frac{1}{3}$, and the sum $S_2$ of the multiplicities of the spectral numbers $\alpha > \frac{4}{3}$. By the symmetry property of the spectrum, the last case can be replaced by $\alpha < \frac{2}{3}$. \\
The first sum $S_1$ is at least equal to
$$S_1 = 1 + 2 + \ldots + (m'-1)=\frac{m'(m'-1)}{2}.$$
The second sum $S_2$ is at least equal to
$$S_2 = 1 + 2 +\ldots + (2m'-1)=m'(2m'-1).$$
It follows that
\begin{equation}
\label{E7}
{\rm deg}_B \, {\rm Sp}(X,0) \leq (m-1)^2 - S_1 - S_2 =(m-1)^2-\frac{m'(5m'-3)}{2}.
\end{equation}
Therefore the semicontinuity theorem implies that
\begin{equation}
\label{E8}
n_2 + 3t + 4n_3 \leq (m-1)^2-\frac{m'(5m'-3)}{2}.
\end{equation}
Observe that the combinatorial count (\ref{comb:naiv}) can be rewritten as
\begin{equation}
\label{E4}
n_2 + 2t + 3n_3 = \binom{m}{2} - k.
\end{equation}
By the above, we can conclude that
$$t + n_3 \leq (m-1)^2- \binom{m}{2} + k -\frac{m'(5m'-3)}{2}=$$
$$= \binom{m-1}{2} + k -\frac{m'(5m'-3)}{2}.$$
For the second inequality, we choose the semicontinuity domain $B=(-\frac{1}{3},\frac{2}{3}]$, and for this choice of $B$ we have
\begin{equation}
\label{E9}
{\rm  deg}_B \, \sum_y {\rm Sp} (X_s,y) = n_3,
\end{equation}
where the sum is taken over all the singular points $y \in X_s$. On the other hand, we have
\begin{equation}
\label{E6}
{\rm deg}_B \, {\rm Sp}(X,0) \leq 1 + 2+ \ldots + (2m'+1)=(m'+1)(2m'+1).
\end{equation}
This completes the proof.
\end{proof}
\begin{ex}
\label{exH1}
These bounds are rather good, at least in some cases. In order to see this for the bound involving $t+n_3$, consider Figure \ref{con-lin} where we present a conic-line arrangement with $d=3$ and $k=2$ having $t=5$ and $n_3=3$. In this case $m=7$, hence $m'=2$ and the first inequality in Theorem \ref{prop2} is
$$8=t+n_3 \leq 10.$$
Next, consider the dual Hesse arrangement given by
$$(x^3-y^3)(y^3-z^3)(x^3-z^3)=0,$$
which has $n_3=12$ triple points. In this case $m=9$, $m'=2$, and the second inequality in Theorem \ref{prop2} gives us
$$12=n_3 \leq 15.$$
\end{ex}

\begin{rk}
\label{rk2}
Since $m'=(m-\epsilon)/3$ and $k \leq m/2$, it follows that we have
$$t+n_3 \leq \frac{1}{18}\bigg(4m^2+m(10\epsilon-9)-5\epsilon^2-9\epsilon+18\bigg)\approx \frac{2}{9}m^{2} + O(m).$$
\end{rk}

\section{Combinatorial constraints on the freeness of reduced curves}
We begin with a general introduction to the subject. Let $C$ be a reduced curve $\mathbb{P}^{2}_{\mathbb{C}}$ of degree $m$ given by $f \in S :=\mathbb{C}[x,y,z]$. We denote by $J_{f}$ the Jacobian ideal generated by the partials derivatives $\partial_{x}f, \, \partial_{y}f, \, \partial_{z}f$. Moreover, we denote by $r:={\rm mdr}(f)$ the minimal degree of a relation among the partial derivatives, i.e., the minimal degree $r$ of a triple $(a,b,c) \in S_{r}^{3}$ such that 
$$a\cdot \partial_{x} f + b\cdot \partial_{y}f + c\cdot \partial_{z}f = 0.$$
We denote by $\mathfrak{m} = \langle x,y,z \rangle$ the irrelevant ideal. Consider the graded $S$-module $N(f) = I_{f} / J_{f}$, where $I_{f}$ is the saturation of $J_{f}$ with respect to $\mathfrak{m}=\langle x,y,z\rangle$.
\begin{definition}
We say that a reduced plane curve $C$ is \emph{free} if $N(f) = 0$. 
\end{definition}
Let us recall that for a reduced curve $C : f=0$ we define the Arnold exponent $\alpha_{C}$ which is the minimum of the Arnold exponents of the singular points $p$ in $C$. Using the modern language, the Arnold exponents of singular points are nothing else than the log canonical thresholds of singularities.
\begin{definition}
Let $C \, : \, f = 0$ be a reduced curve in $\mathbb{C}^2$ which is singular at $0 \in \mathbb{C}^{2}$. Denote by $\phi : Y \rightarrow \mathbb{C}^{2}$ the standard minimal resolution of singularities, i.e., the smallest resolution that has simple normal crossings (which exists and it is unique). We write then $K_{Y} = \phi^{*} K_{\mathbb{C}^{2}}+\sum_{i} a_{i} E_{i}$ and $\phi^{*}C = \phi_{*}^{-1} C + \sum_{i} b_{i} E_{i}$, where $=$ means the linear equivalence. Then the log canonical threshold of $C$ in $\mathbb{C}^{2}$ is defined as
$$c_{0}(f) = {\rm min}_{i} \bigg\{ \frac{a_{i}+1}{b_{i}}\bigg\}.$$
\end{definition}
Using this local (analytical) description, the Arnold exponent $\alpha_{C}$ of $C$ is then the minimum over all log canonical thresholds of singular points. 
In order to compute the actual values of the log canonical thresholds, we can us the following result -- see for instance \cite[Theorem 4.1]{Cheltsov}. 
\begin{thm}
Let $C$ be a reduced curve in $\mathbb{C}^{2}$ which has degree $m$. Then $c_{0}(f) \geq \frac{2}{m}$, and the equality holds if and only if $C$ is a union of $m$ lines passing through $0$.
\end{thm}
\begin{rk}
If $p=(0,0) \in \mathbb{C}^{2}$ is an ordinary singularity of multiplicity $r$ determined by $C \, : \, f=0$, then $c_{0}(f) = \frac{2}{r}$.
\end{rk}

Now we need to compute the log canonical threshold for tacnodes. Since tacnodes are quasi-homogeneous singularities, then we can use the following pattern (cf. \cite[Formula 2.1]{DimcaSernesi}).

Recall that the germ $(C,p)$ is weighted homogeneous of type $(w_{1},w_{2};1)$ with $0 < w_{j} \leq 1/2$ if there are local analytic coordinates $y_{1}, y_{2}$ centered at $p=(0,0)$ and a polynomial $g(y_{1},y_{2})= \sum_{u,v} c_{u,v} y_{1}^{u} y_{2}^{v}$ with $c_{u,v} \in \mathbb{C}$, where the sum is taken over all pairs $(u,v) \in \mathbb{N}^{2}$ with $u w_{1} + v w_{2}=1$. In this case, we have $$c_{0}(g) = w_{1}+w_{2}.$$
\begin{rk}
Let $g = y^{2} + x^{4}$, so $g$ defines a tacnode at $p=(0,0)$. Then $w_{1}=\frac{1}{2}, w_{2} = \frac{1}{4}$, and hence we have $c_{0}(g) = \frac{3}{4}$
\end{rk}
In order to show our main result for this section, recall the following \cite[Theorem 2.1]{DimcaSernesi}.
\begin{thm}[Dimca-Sernesi]
Let $C \, : \, f = 0$ be a reduced curve of degree $m$ in $\mathbb{P}^{2}_{\mathbb{C}}$ having only quasi-homogeneous singularities. Then $${\rm mdr}(f) \geq \alpha_{C}\cdot m -2.$$
\end{thm}
Since nodes, tacnodes, and ordinary triple points are quasi-homogeneous singularities, then we can prove the following result.
\begin{prop}
\label{prop:bound9}
Let $C \, : \, f=0$ be a reduced curve of degree $m$ in $\mathbb{P}^{2}_{\mathbb{C}}$ having only nodes, tacnodes, and ordinary triple points as singularities. Then
$${\rm mdr}(f) \geq \frac{2}{3}m-2.$$
In particular, if $C$ is free, then $m \leq 9$.
\end{prop}
\begin{proof}
Since $C$ has nodes, tacnodes, and ordinary triple points as singularities, then
$$\alpha_{C} = {\rm min} \bigg\{1, \frac{3}{4}, \frac{2}{3} \bigg\} = \frac{2}{3},$$
so by the above result we have
$${\rm mdr}(f) \geq \frac{2}{3}m - 2.$$ 
If $C$ is a free curve, then
$$\frac{2}{3}m - 2 \leq {\rm mdr}(f) \leq \frac{m-1}{2}, $$
which gives $m \leq 9$. 
\end{proof}
\begin{rk}
If we restrict our attention to reduced free curves with nodes and tacnodes, then analogous computations as above give $m \leq 5$, and this bound is sharp according to what we shall see in Example \ref{ex5.5} in the forthcoming section. In fact, we can show that for every $m \in \{3,4,5\}$ there exists a conic-line arrangement having nodes and tacnodes with $k\geq 1$ that is free.
\end{rk}
\begin{rk}
Observe that Proposition \ref{prop:bound9} is sharp in the class of reduced free curves -- the dual Hesse arrangement of $9$ lines and $12$ triple points considered in Example \ref{exH1} is free.
\end{rk}
Let $\mathcal{CL} = \{\ell_{1}, ..., \ell_{d}, C_{1}, ..., C_{k}\} \subset \mathbb{P}^{2}_{\mathbb{C}}$ be a free arrangement of $d \geq 1$ lines and $k\geq 1$ conics. We are going to use the following homological characterization of the freeness, see for instance \cite{DimcaSticlaru}, which can be checked on specific examples using \verb{Singular{ \cite{Singular}, or other computer algebra software.
\begin{thm}
 \label{thmF}
Let $C \subset \mathbb{P}^{2}_{\mathbb{C}}$ be a reduced curve of degree $m$ and let $f=0$ be its defining equation. Then $C$ is free if and only if the minimal free resolution of the Milnor algebra $M(f) = S/J_{f}$ has the following form:
\begin{equation*}
0 \rightarrow S(-d_{1}-(m-1)) \oplus S(-d_{2}-(m-1)) \rightarrow S^{3}(-m+1)\rightarrow S \rightarrow M(f) \rightarrow 0
\end{equation*}
with $d_{1} + d_{2} = m-1$.  In particular, if $d_1 \leq d_2$, then ${\rm mdr}(f)=d_1 \leq \frac{m-1}{2}$. \end{thm}

We will need additionally the following lemma,  see for instance \cite[Lemma 4.4]{DimcaSernesi}.
\begin{lem}
 \label{lemF}
If $C$ is a free plane curve of degree $m$, then the exponents $(d_{1},d_{2})$ are positive integers satisfying the following system of equations:
$$d_{1}+d_{2} = m-1, \quad \quad d_{1}d_{2} = (m-1)^{2} - \tau(C),$$
where $\tau(C)$ denotes the total Tjurina number.
\end{lem}
If now $\mathcal{CL}$ is a conic-line arrangement having degree $m = 2k + d$ with $n_{2}$ nodes, $t$ tacnodes, and $n_{3}$ ordinary triple points, then the above lemma can be rewritten as
\begin{equation}
\label{eq:conlinfree}
d_{1} + d_{1} = m-1, \quad \quad d_{1}^{2} + d_{2}^{2} + d_{1}d_{2} = n_{2} + 3t + 4n_{3}, \quad \quad d_{1} \leq d_{2}.
\end{equation}
So our problem reduces to possible geometrical realizations of some positive integer solutions to (\ref{eq:conlinfree}). 
\begin{ex}
\label{ex5.3}
Let $m=3$, so we have one conic and one line. An easy inspection tells us that $d_{1} = d_{2}=1$, $n_{2} = n_{3} = 0$, and $t=1$ satisfies (\ref{eq:conlinfree}). Let us consider
$$\mathcal{CL}_{3} : \quad (x-z)\cdot (x^{2} + y^{2} - z^{2}) = 0.$$
Observe that $\mathcal{CL}_{3}$ is exactly a line tangent to a conic.
We can compute the minimal free resolution of the Milnor algebra $M(f) = S/J_{f}$ which has the following form:
$$0 \rightarrow S^{2}(-3) \rightarrow S^{3}(-2) \rightarrow S \rightarrow M(f) \rightarrow 0,$$
so $\mathcal{CL}_{3}$ is free with the exponents $(1,1)$.
\end{ex}
\begin{ex}
\label{ex5.4}
Consider now the case with $m=4$, and note that there are many solutions to (\ref{eq:conlinfree}). Take $d_{1} = 1$ and $d_{2}=2$, then we can find the following Diophantine solution, namely $t=2$ and $n_{2}=1$.  Consider
$$\mathcal{CL}_{4} : \quad (x^{2}-z^{2})\cdot (x^{2} + y^{2} - z^{2}) = 0,$$
hence two lines tangent to the same conic.
Then the minimal free resolution of the Milnor algebra $M(f) = S/J_{f}$ has the following form:
$$0 \rightarrow S(-5)\oplus S(-4) \rightarrow S^{3}(-3) \rightarrow S \rightarrow M(f) \rightarrow 0,$$
so $\mathcal{CL}_{3}$ is indeed free with the exponents $(1,2)$.
\end{ex}

\begin{ex}
\label{ex5.5}
Consider now the case with $m=5$. One positive integer solution to  (\ref{eq:conlinfree}) is $d_{1}=d_{2}=2$ and $n_{2}=t=3$, $n_3=0$. Consider the following conic-line arrangement 
$$\mathcal{CL}_{5} : (y-z)\cdot(x^{2}-z^{2})\cdot(x^{2}+y^{2}-z^{2})=0.$$
The minimal free resolution of the Milnor algebra $M(f) = S/J_{f}$ has the following form:
$$0 \rightarrow S^{2}(-6) \rightarrow S^{3}(-4) \rightarrow S \rightarrow M(f) \rightarrow 0,$$
so $\mathcal{CL}_{5}$ is free. \\
Another positive integer solution to  (\ref{eq:conlinfree}) is $d_{1}=d_{2}=2$ and $n_{2}=t=0$, $n_3=3$.
Consider now the second case and the following arrangement 
$$\mathcal{CL}_{5}^{'} : y\cdot (x+y-4z)\cdot (x-y+4z)\cdot (x^{2}+y^{2}-16z^{2})=0.$$
Observe that $\mathcal{CL}_{5}'$ has $n_{2} = t=0$ and $n_{3}=3$.
The minimal free resolution of the Milnor algebra $M(f) = S/J_{f}$ has the same form as above, namely
$$0 \rightarrow S^{2}(-6) \rightarrow S^{3}(-4) \rightarrow S \rightarrow M(f) \rightarrow 0,$$
so $\mathcal{CL}_{5}^{'}$ is free.  
\end{ex}
\begin{figure}[ht]
\begin{center}
\begin{minipage}[t]{0.45\textwidth}
\begin{tikzpicture}[line cap=round,line join=round,>=triangle 45,x=1cm,y=1cm,scale=0.65]
\clip(-4.952222222222223,-2.9022222222222234) rectangle (5.252222222222216,3.462222222222221);
\draw [line width=2pt] (0,0) circle (1cm);
\draw [line width=2pt] (1,-2.9022222222222234) -- (1,3.462222222222221);
\draw [line width=2pt,domain=-4.952222222222223:5.252222222222216] plot(\x,{(--1.1547005383792515--0.5773502691896257*\x)/1});
\draw [line width=2pt,domain=-4.952222222222223:5.252222222222216] plot(\x,{(-1.1547005383792515-0.5773502691896257*\x)/1});
\end{tikzpicture}

\end{minipage}
\begin{minipage}[t]{0.45\textwidth}
\begin{tikzpicture}[line cap=round,line join=round,>=triangle 45,x=1cm,y=1cm,scale=0.25]
\clip(-11.608637835919069,-6.777391390325404) rectangle (14.914385982853624,9.519628926891263);
\draw [line width=2pt] (0,0) circle (4cm);
\draw [line width=2pt,domain=-11.608637835919069:14.914385982853624] plot(\x,{(-0-0*\x)/1});
\draw [line width=2pt,domain=-11.608637835919069:14.914385982853624] plot(\x,{(--4-1*\x)/1});
\draw [line width=2pt,domain=-11.608637835919069:14.914385982853624] plot(\x,{(-4-1*\x)/-1});
\end{tikzpicture}
\end{minipage}
\end{center}
\caption{Free conic-line arrangements with $m=5$ and $k=1$.}
\label{free5}
\end{figure}
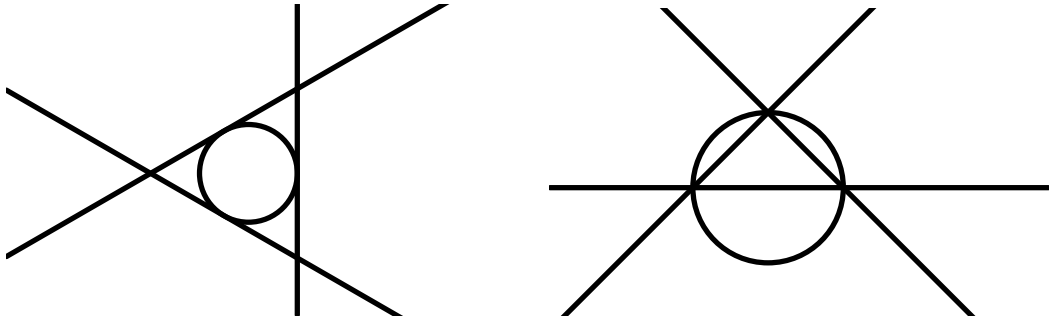 

\section{Classification of free conic-line arrangements with nodes, tacnodes and triple points}

After the above warm-up, we are ready to present the first classification result. It gives a complete characterization of free arrangements having $d \in \{1,2,3\}$ lines and $k \geq 1$ smooth conics. We start with a general discussion about free conic-line arrangements with nodes, tacnodes, and ordinary triple points. \\\\
If $\mathcal{CL}$ is free with the exponents $(d_1,d_2)$, then we must have
$$d_1+d_2=2k+d-1$$
and
$$d_1d_2=(2k+d-1)^2-n_2-3t-4n_3=(2k+d-1)^2-(n_2+ 2t + 3n_3) - t - n_3.$$
Using the combinatorial count (\ref{comb:naiv}), we obtain
\begin{equation}
\label{eq:exp}
d_1d_2 = 2k^2-2k+1 + 2kd +\frac{d^2-3d}{2} - t - n_{3}.   
\end{equation}
If we multiply equation (\ref{eq:exp}) by $2$, add it to (\ref{comb:naiv}), then we obtain
$$2d_1d_2=2k^2+2k(d-1)+\frac{(d-1)(d-4)}{2}+n_2+n_3.$$
Since $d_1d_2 \leq (2k+d-1)^2/4$, we get
\begin{equation}
\label{E7"}
n_2+n_3 \leq \frac{3}{2}(d-1).
\end{equation}
Using once again $(\ref{comb:naiv})$, we obtain
\begin{equation}
\label{E7'}
t \geq k^2 - k + kd +\frac{(d-1)(d-3)}{4}-n_3.
\end{equation}
If $k \geq 2$, two conics $C_1$ and $C_2$ from the arrangement $\mathcal{CL}$ can be in one of the following $3$ situations:
\begin{enumerate}

\item $|C_1 \cap C_2|=4$, and then all the intersection points are nodes. 

\item  $|C_1 \cap C_2|=3$, and then one intersection point produces a tacnode in $\mathcal{CL}$, and the other two intersection points will give nodes.

\item  $|C_1 \cap C_2|=2$, and then two intersection points produce two tacnodes in $\mathcal{CL}$.
\end{enumerate}
Let $m_j$ with $j \in\{2,3,4\}$ be the number of pairs of conics in $\mathcal{CL}$ such that $|C_1 \cap C_2|=j$. Then the number of tacnodes coming from the contact of $2$ conics is 
\begin{equation}
\label{E16a}
t'=2m_2+m_3.
\end{equation}
We also have, by counting pairs of conics in two different ways,
\begin{equation}
\label{E16b}
m_2+m_3+m_4=\binom{k}{2}
\end{equation}
and 
\begin{equation}
\label{E16c}
2m_3+4m_4\leq n_2+n_3.
\end{equation}
The last inequality follows from the fact that the nodes coming from $C_1 \cap C_2$ will give either nodes or triple points in $\mathcal{CL}$. To evaluate the number of tacnodes created by a line in $\mathcal{CL}$, we need the following.
\begin{lem}
\label{lem16}
Let $C_1$ and $C_2$ be two smooth conics in $\mathcal{CL}$ such that 
$|C_1 \cap C_2|=2$. Then any line $L$ in $\mathcal{CL}$ is tangent to at most one of the conics $C_1$ and $C_2$.
\end{lem}
\proof
Since the pair of conics $C_1$ and $C_2$ gives rise to two tacnodes, then up-to a linear change of coordinates one can take
$$C_1:x^2+y^2-z^2=0 \text{ and } C_2: x^2+y^2-r^2z^2=0$$
where $r \in \C$, $r \ne 0, \pm 1$ -- please consult \cite[Proposition 3]{Meg}.
It follows that the dual $C_1^{\vee}$ is given by $x^2+y^2-z^2=0$, while the dual  $C_2^{\vee}$ is given by $x^2+y^2-\frac{1}{r^{2}}z^2=0$.
A common tangent for $C_1$ and $C_2$ corresponds to a point in the
intersection $C_1^{\vee} \cap C_2^{\vee}$, and hence it is given by equation $T_{\pm}: x\pm y=0$. However, both these lines pass through one of the two tacnodes situated at $(1:\pm 1:0)$. Hence $L$ cannot be tangent to two conics, since the curve $\CC$ has only  nodes, ordinary triple points and tacnodes as singularities.
\endproof

\begin{thm}
\label{thm1}
Let $\mathcal{CL}$ be an arrangement of $0 \leq d \leq 3$ lines and $k \geq 1$ smooth conics with $n_2$ nodes, $t$ tacnodes, and $n_3$ ordinary triple points. Assume that  $\mathcal{CL}$  is free, then the following pairs are admissible: $$(d,k) \in \{(1,1), (2,1), (3,1), (3,2)\}.$$
\end{thm}
\begin{proof}
We need to consider some cases.

{\bf Case $d=0$.} Then \eqref{E7"} implies $n_2+n_3 < 0$, which is clearly impossible.

{\bf Case $d=1$.}  Then $n_2+n_3 \leq 0$, and hence $n_2=n_3=0$.
Using the combinatorial count (\ref{comb:naiv}) we get $t=k^2$. The case $k=1$ is clearly possible, i.e., a conic plus a tangent line form a free curve, see Example \ref{ex5.3}. We show now that the case $k > 1$ is impossible. Using \eqref{E16b}, we see that $m_3=m_4=0$, and hence any two conics in $\mathcal{CL}$ meet in two points, as in Lemma \ref{lem16} above.

There are $\binom{k}{2}$ pairs of conics, hence the number of tacnodes obtained as intersection of two conics is $t'=k^2-k$. Recall also that by \cite{DJP} one has $k \in \{2,3,4\}$. The only possibility to have $t=k^2$ tacnodes is that the unique line, call it  $L$, is tangent to all the conics simultaneously. In the light of Lemma \ref{lem16}, $L$ cannot be tangent to each conic in $\mathcal{CL}$. This completes the proof in this case.

{\bf Case $d=2$.} The case $k=1$ is possible, see Example \ref{ex5.4}. We show that the cases $k \geq 2$ are impossible. One has  $n_2+n_3 \leq 1$ and using (\ref{E7'}) above, we get
$$t \geq k^2-k+(2k-\frac{1}{4}-n_3).$$
When $k \geq 2$, one has $2k-\frac{1}{4}-n_3 >2$, and hence
$$t \geq k^2-k+3.$$
Combining the condition $n_{2} + n_{3} \leq 1$ and inequality \eqref{E16c}, we get $m_3 = m_4 =0$. It means, in particular, that $m_{2} = \binom{k}{2}$, $t' = k^{2} - k$, and $k \in \{2,3\}$. An easy inspection, performed along the lines of Lemma \ref{lem16}, shows that a line can add at most one tacnode, and in this case one must have 
$$t \leq k^2-k+2,$$
hence we have a contradiction.

{\bf Case $d=3$.} Observe that by Proposition \ref{prop:bound9} we have $k \in \{1,2,3\}$. The case $k=1$ is possible, see Example \ref{ex5.5}.
 We show now that  $k = 3$ is impossible. In this case one has
$n_2+n_3\leq 3$ and using formula (\ref{E7'}) above, we get
\begin{equation}
\label{E16d}
t \geq k^2-k+(3k-n_3).
\end{equation}
Then \eqref{E16c} and $n_{2} + n_{3} \leq 3$ lead us to $m_3\leq 1$ and $m_4=0$. If $m_3=0$, then we use Lemma \ref{lem16} to get a contradiction since
$3k-n_3 \geq 6 > d=3$. Assume that $m_3=1$, so let's say that $|C_1 \cap C_2|=3$. The maximal number of tacnodes coming from the contact of $3$ lines and conics in such an arrangement is $5$. Indeed, two lines can be tangents to both $C_1$ and $C_2$, giving $4$ tacnodes, and the third line can be tangent to at most one conic in $\mathcal{CL}$.
It follows that
$$t \leq 2m_2+m_3+5=k(k-1)+4.$$
Since $k = 3$ one has $9 - n_3\geq 6 > 4$, and this contradiction proves the claim.

Consider the remaining case $k=2$. First we assume that $m_3=0$. In order the have
$2d-n_3 \leq d$, we must have $n_3=3$ and hence $n_2=0$. However, any line, say $L_1$, is tangent to one of the two conics, say to $C_1$, and it is secant to the other, so $L \cap C_2=\{p,q\}$. The point $p$ cannot be a node, so one of the remaining lines, say $L_2$, passes through $p$ and is tangent to $C_1$, and the other line, say $L_3$, is passing through $q$, and is again tangent to $C_1$. The third triple point should be
$$r= C_2 \cap L_2 \cap L_3.$$
This configuration is geometrically realizable, we can take for instance
$$\mathcal{CL}_{7}: (x^2+y^2-z^2)\cdot(x^2+y^2-4z^2)\cdot(x-z)\bigg(y + \frac{\sqrt{3}}{3}x + \frac{2\sqrt{3}}{3}z\bigg)\cdot\bigg(y - \frac{\sqrt{3}}{3}x - \frac{2\sqrt{3}}{3}z\bigg)=0.$$
Moreover, it is unique up to a projective transformation. Indeed, we use again \cite[Proposition 3]{Meg} which implies that one may assume that the two conics $C_1$ and $C_2$ are concentric circles, $C_1$  with radius $r_1=1$, and $C_2$  with radius $r_2 \in \C$, $r_2 \ne 0, \pm 1$.
Then we note that  a triangle in which the center of the inscribed circle coincides with the center of the circumscribed circle is necessarily equilateral. This implies that $r_2=2$, and the corresponding equation is given above, where the vertices of the equilateral triangle in the affine plane $z=1$ are $(1,0)$ and $(- \frac{1}{2}, \pm \frac{\sqrt 3}{2})$.
Using \verb{Singular{ we can compute the minimal free resolution of the Milnor algebra of $\mathcal{CL}_{7}$, it has the following form
$$0 \rightarrow S^{2}(-9) \rightarrow S^{3}(-6) \rightarrow S \rightarrow M(f) \rightarrow 0,$$
so $\mathcal{CL}_{7}$ is free.

Finally, we consider the case $k=2$ and $m_3=1$, i.e., the two conics are tangent in one point.  In order to satisfy \eqref{E16d}, we have to use the 3 lines to create many tacnodes and many triple points, in fact we need
$t+n_3 \geq 8$ of such singular points. We can use the first two lines to get two tangents to both conics, hence $4$ new tacnodes. The third line can create either $2$ triple points and $2$ double points, one triple point and $4$ nodes, or a tacnode and $4$ nodes. Neither possibility satisfies $t+n_3 \geq 8$, and this completes the proof.
\end{proof}
\begin{figure}[ht]
\begin{center}
\begin{tikzpicture}[line cap=round,line join=round,>=triangle 45,x=1cm,y=1cm,scale=0.8]
\clip(-4.952222222222223,-2.9022222222222234) rectangle (5.252222222222216,3.462222222222221);
\draw [line width=2pt] (0,0) circle (1cm);
\draw [line width=2pt] (0,0) circle (2cm);
\draw [line width=2pt] (1,-2.9022222222222234) -- (1,3.462222222222221);
\draw [line width=2pt,domain=-4.952222222222223:5.252222222222216] plot(\x,{(--1.1547005383792515--0.5773502691896257*\x)/1});
\draw [line width=2pt,domain=-4.952222222222223:5.252222222222216] plot(\x,{(-1.1547005383792515-0.5773502691896257*\x)/1});
\end{tikzpicture}
\end{center}
\caption{The unique free arrangement of $2$ conics and $3$ lines with $t=5$ and $n_{3}=3$.}
\label{con-lin}
\end{figure}
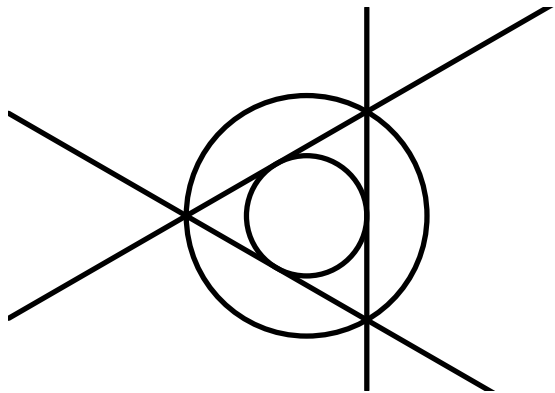
Now we are going to discuss non-freeness for conic-line arrangements with $d \in \{4,5,6,7\}$ -- please bear in mind that the upper bound on the number of lines follows from Proposition \ref{prop:bound9}.
\begin{prop}
\label{prop11a}
Let $\mathcal{CL}$ be an arrangement of $d = 4$ lines and $k \geq 1$ smooth conics having only nodes, tacnodes, and ordinary triple points as singularities. Then $\mathcal{CL}$ is never free.
\end{prop}
\begin{proof}
If we assume that $\CC$ is free, then formula (\ref{E7"}) implies that
\begin{equation}
\label{E11a}
n_2+ n_3 \leq 4
\end{equation}
and by (\ref{E7'}) we obtain 
\begin{equation}
\label{E11b}
t \geq k^2-k+4k + \frac{3}{4} - n_3.
\end{equation}
Using \eqref{E16c}, it follows that
\begin{equation}
\label{E11b1}
2m_3+4m_4 \leq 4.
\end{equation}
In particular, we have $m_3+m_4 \leq 2$. Consider the graph $\Gamma(\CC)$ associated to the conics in $\CC$, whose vertices are these conics, and two vertices $C_i$ and $C_j$ are connected by an edge (resp. by a double edge) if and only if $|C_i \cap C_j|=3$ (resp.
$|C_i \cap C_j|=4$) -- please consult \cite{Meg} for more details. It follows that this graph has at most two edges. To get the maximum number of tacnodes, Lemma \ref{lem16} implies that the $4$ lines have to be tangent lines to the conics connected by an edge to other conic. There are either $3$ conics in a chain connected by edges, or two pairs of conics, each pair connected by an edge.
A simple analysis show that the maximum number of tacnodes created in this way is $8$. Hence the number of tacnodes satisfies
\begin{equation}
\label{E11b1'}
t \leq t'+8=k^2-k-m_3-2m_4+8 \leq k^2-k+8.
\end{equation}
Combining this with inequality \eqref{E11b} we get
$$4k-3 \leq 4k+1-n_3 \leq 8, $$
which implies $k \leq 2$. 
\medskip

{\bf Case: $k=2$.}

\medskip

If $m_2=1$, then any line is tangent to at most one conic, by  Lemma \ref{lem16}, hence we have $t \leq 6$ and \eqref{E11b} yields a contradiction.

If $m_3=1$, the inequality \eqref{E11b} implies $t + n_3 \geq 11$.
We can use $2$ lines, say $L_1$ and $L_2$ to create $4$ new tacnodes ($5$ in total) and a new node ($3$ in total). The third line $L_3$ can produce $2$ triple points if it passes through the $2$ nodes in the intersection $C_1 \cap C_2$, but it will create $2$ new nodes as the intersections with $L_1$ and $L_2$. This stands in contradiction with \eqref{E11a}. If $L_3$ passes through the node $L_1\cap L_2$, then it will be secant to both $C_1$ and $C_2$, thus creating $4$ new nodes, again it stands in contradiction with \eqref{E11a}.
Finally, if $L_3$ is tangent to one of the conics, it would be secant for the other, creating $4$ new nodes, $2$ on that second conic and $2$ on $L_1 \cap L_2$, a contradiction with respect to \eqref{E11a}.

If $m_4=1$, $C_1\cap C_2$ gives rise to $4$ nodes. As soon as we add $2$ lines, a new node will be created (and some old nodes transformed in triple points), but in any case we get a contradiction with \eqref{E11a}.

Hence the case $k=2$ is impossible.

\medskip

{\bf Case: $k=1$.}

\medskip

Since $k=1$, we clearly have $t \leq 4$.
The only possibility to have \eqref{E11b} satisfied is that $n_3 \geq 1$.
If we look on the (sub)arrangement $\mathcal{A}$ of the $4$ lines in $\mathcal{CL}$, we have $2$ cases to discuss.

\medskip

{\bf Case 1: $\mathcal{A}$ has $3$ concurrent lines.}

\medskip

Let $L_1,L_2$ and $L_3$ be the lines meeting at a point $p$, and $L_4$ is a secant line, meeting $L_j$ at $q_j$ for $j=1,2,3$.

If $n_3=1$, we get from \eqref{E11b} that $t \geq 4$, which is impossible. Indeed, only two lines among $L_1,L_2$ and $L_3$ can be tangent to the conic, so we get $t \leq 3$.

If $n_3=2$, the second triple point is at the intersection of a secant line to the conic passing through $p$, say $L_2$, the conic and the line $L_4$. In this case, we get from \eqref{E11b} that $t \geq 3$, which is impossible, since $L_2$ and $L_4$ cannot be tangent to the conic.

If $n_3=3$, the new triple point is at the intersection of a new secant line to the conic passing through $p$, say $L_1$, the conic and the line $L_4$. In this case, we get from \eqref{E11b} that $t \geq 2$, which is impossible, since $L_1$, $L_2$ and $L_4$ cannot be tangent to the conic.

Finally, if $n_3=4$, then the 3 triple points distinct from $p$ are all situated on the conic and on the line $L_4$, which is impossible.

\medskip

{\bf Case 2: $\mathcal{A}$ is a nodal arrangement.}

\medskip

It follows that now all the triple points are on the conic. More precisely, the conic passes through $n_3$ nodes of the line arrangement $\mathcal{A}$. Hence $\mathcal{CL}$ has $n_3$ triple points and at least $6-n_3$ double points. It follows that $n_2+n_3 \geq 6$, a contradiction with \eqref{E11a}.
\end{proof}

\begin{prop}
\label{prop11b}
Let $\mathcal{CL}$ be an arrangement of $d=5$ lines and $k \geq 1$ smooth conics having only nodes, tacnodes, and ordinary triple points. Then $\mathcal{CL}$  is never free.
\end{prop}
\begin{proof}
If we assume that $\mathcal{CL}$ is free. Using Proposition \ref{prop:bound9}, we get $m = 2k+5 \leq 9$, and hence $k\leq 2$.
Then formula (\ref{E7"}) implies that
\begin{equation}
\label{E11a2}
n_2+ n_3 \leq 6
\end{equation}
and by (\ref{E7'}) we get
\begin{equation}
\label{E11b2}
t \geq k^2-k+5k+2-n_3.
\end{equation}

\medskip

{\bf Case $k=2$.}

\medskip

If $m_2=1$, then any line is tangent to at most one conic, which follows from Lemma \ref{lem16}, so we have $t \leq 7$ and \eqref{E11b2} yields a contradiction.

If $m_3=1$, the inequality \eqref{E11b2} implies $t + n_3 \geq 14$.
We can use $2$ lines, say $L_1$ and $L_2$, to create $4$ new tacnodes ($5$ in total) and a new node ($3$ in total). The third line $L_3$ can produce $2$ triple points if it passes through the $2$ nodes in the intersection $C_1 \cap C_2$, but will create $2$ new nodes as intersections with $L_1$ and $L_2$. It means that the remaining $2$ lines should not create any new node or triple point, which is impossible.  If $L_3$ passes through the node $L_1\cap L_2$, then it will be secant to both $C_1$ and $C_2$, thus creating 
$4$ new nodes, a contradiction with \eqref{E11a}.
Finally, if $L_3$ is tangent to one of the conics, it would be secant for the other, creating $4$ new nodes, $2$ on that second conic and $2$ on $L_1 \cap L_2$. This contradicts \eqref{E11a}.

If $m_4=1$, $C_1\cap C_2$ gives rise to 4 nodes. As soon as we add $3$ lines, at least $3$ new nodes will be created (and some old nodes transformed in triple points), but in any case we get a contradiction with \eqref{E11a}.

Hence the case $k=2$ is impossible.
\medskip

{\bf Case: $k=1$.}

\medskip
The above shows that the only possibility to have \eqref{E11a2} and \eqref{E11b2} satisfied is that $k=1$ and $n_3 \geq 2$.
If we look on the (sub)arrangement $\mathcal{A}$ of the $5$ lines in $\mathcal{CL}$, there are $3$ cases to discuss.

\medskip

{\bf Case 1: $\mathcal{A}$ has $2$ triple points.}

\medskip

If we denote by $p_1$ and $p_2$ the two triple points, then the line determined by these points must be in $\mathcal{A}$. We denote this line by $L$. The other $2$ lines passing through $p_1$ ($p_2$, respectively) are denoted by $L_1$ and $L_1'$ ($L_2$ and $L_2'$, respectively). There are $4$ nodes in the arrangement $\mathcal{A}$, the intersections $L_1 \cap L_2$, $L_1 \cap L_2'$,
$L_1' \cap L_2$, and $L_1' \cap L_2'$.

If $n_3=2$, we get from \eqref{E11b2} that $t \geq 5$, which is impossible since $L$, $L_1$ and $L_1'$ cannot be all tangent to the conic.

If $n_3=3$, the new triple point $q_1$ is at the intersection of two secant lines with the conic. In this case, we get from \eqref{E11b} that $t \geq 4$, which is impossible, since the $2$ lines meeting at $q_1$ cannot be tangent to the conic.

If $n_3=4$, then there are $2$ triple points $q_1$ and $q_2$ situated on the conic, coming from the intersection of at least $3$ lines from $\mathcal{A}$ with the conic. In this case, we get from \eqref{E11b} that $t \geq 3$, which is impossible since the lines passing through $q_1$ or $q_2$ cannot be tangent to the conic.

If $n_3=5$, then there are $3$ triple points $q_1$, $q_2$ and $q_3$ situated on the conic, coming from the intersection of at least $4$ lines from $\mathcal{A}$ with the conic. In this case, we get from \eqref{E11b} that $t \geq 2$, which is impossible since the lines passing through $q_1$, $q_2$ or $q_3$ cannot be tangent to the conic.

Finally, if $n_3=6$, then $n_2=0$, the conic, call it $Q$, passes through all the $4$ nodes of $\mathcal{A}$ and it is tangent to the line $L$. The conics passing through the $4$ nodes form a pencil of conics, determined by the degenerate conics $Q_1$ and $Q_2$, where $Q_1$ (resp. $Q_2$) is  the union $L_1 \cup L_1'$ (resp. $L_2 \cup L_2'$).
Note that these two degenerate conics $Q_1$ and $Q_2$ meet the line $L$ in one point, namely $Q_1\cap L=p_1$ and $Q_2\cap L=p_2$.
The conic $Q$ is in this pencil $\al Q_1 + \be Q_2$, and meets
the line $L$ in one point as well. This is a contradiction, since in a pencil only two members can meet a given line in a single point. To check this claim, we assume that $L$ is given by $x=0$. Then the condition that $\al Q_1 + \be Q_2$ meets
line $L$ in one point is a quadratic form in $y,z$, with coefficients being linear forms in $\al, \be$, such that it has zero discriminant, which yields the vanishing of the quadratic form in $\al, \be$.

\medskip

{\bf Case 2: $\mathcal{A}$  has a triple point.}

\medskip

It follows that $\mathcal{A}$ has $7$ nodes. To create $n_3$ triple points in $\mathcal{CL}$, the conic passes through $n_3-1$ nodes of the line arrangement $\A$. Hence $\mathcal{CL}$ has $n_3$ triple points and at least $7-(n_3-1)$ double points. It follows that $n_2+n_3 \geq 8$, a contradiction with respect to \eqref{E11a2}.

\medskip

{\bf Case 3: $\mathcal{A}$ is a nodal arrangement.}

\medskip

It follows that now  $\mathcal{A}$ has $10$ nodes. To create $n_3$ triple points in $\mathcal{CL}$, the conic passes through $n_3$ nodes of the line arrangement $\mathcal{A}$. Hence $\mathcal{CL}$ has $n_3$ triple points and at least $10-n_3$ double points. It follows that $n_2+n_3 \geq 10$, a contradiction with \eqref{E11a2}.

\end{proof}

\begin{prop}
\label{prop11c}
Let $\mathcal{CL}$ be an arrangement of $d = 6$ lines and $k \geq 1$ smooth conics having only nodes, tacnodes, and ordinary triple points. Then  $\mathcal{CL}$ is never free.
\end{prop}
\proof
If we assume that $\mathcal{CL}$ is free, then using Proposition \ref{prop:bound9} we get $m = 2k + 6 \leq 9$, and hence $k=1$.
Then the formula (\ref{E7"}) implies that
\begin{equation}
\label{E11a3}
n_2+ n_3 \leq 7
\end{equation}
and by (\ref{E7'}) we have
\begin{equation}
\label{E11b3}
t \geq 10-n_3.
\end{equation}
Since $t \leq 6$, the only possibility to have \eqref{E11a3} and \eqref{E11b3} been satisfied is $n_3 \geq 4$.
Consider the (sub)arrangement $\mathcal{A}$ of $6$ lines in $\mathcal{CL}$. Then $\mathcal{A}$ has only double and triple points, and let us denote by $n_2'$ and $n_3'$ their respective numbers. It is known that 
\begin{equation}
\label{E11d3}
n_2'+3n_3' =\binom{6}{2}=15.
\end{equation}
The conic in $\mathcal{CL}$ has to pass through $n_3-n_3'$ nodes of $\mathcal{A}$, and hence $n_2 \geq n_2'-(n_3-n_3')$. It follows that
$$n_2+n_3 \geq n_2'+n_3' =15 - 2n_3'.$$
It is well-known that the maximal number of triple points $n_3'$ in this case is $4$, and the arrangement is projectively equivalent to
$$\A_0: (x^2-y^2)(x^2-z^2)(y^2-z^2)=0.$$
Combining this fact with inequality in \eqref{E11a3}, it follows that
we are exactly in this case, that is $n_2' = 3$ and $n_3' = 4$. Moreover, the intersection between the conic and the lines in $\mathcal{A}$ should not add any new point to the set of seven multiple points of $\mathcal{A}$. Note that each line contains exactly one double point. This would imply that the conic is tangent to the 6 lines (in new points given rise to tacnodes), but this is clearly impossible as we have seen above ($3$ concurrent lines cannot all be tangent to the same conic).
\endproof

\begin{prop}
\label{prop11d}
Let $\mathcal{CL}$ be an arrangement of $d = 7$ lines and $k \geq 1$ smooth conics having only nodes, tacnodes, and ordinary triple points as singularities. Then $\mathcal{CL}$ cannot free.
\end{prop}
\begin{proof}
If we assume that $\mathcal{CL}$ is free, then using Proposition \ref{prop:bound9} we get $m=2k+7 \leq 9$, and hence $k=1$.
Then the formula (\ref{E7"}) implies that
\begin{equation}
\label{E11a4}
n_2+ n_3 \leq 9
\end{equation}
and the formula (\ref{E7'}) gives
\begin{equation}
\label{E11b4}
t \geq 13-n_3.
\end{equation}
Since $t \leq 7$, the only possibility to have \eqref{E11a4} and \eqref{E11b4} been satisfied is $n_3 \geq 6$.
Consider the (sub)arrangement $\mathcal{A}$ of $7$ lines in $\mathcal{CL}$. Then $\mathcal{A}$ has only double and triple points, and let us denote by $n_2'$ and $n_3'$ their respective numbers. By the combinatorial count, we have
\begin{equation}
\label{E11d4}
n_2' + 3n_3' = \binom{7}{2} = 21.
\end{equation}
The conic in $\mathcal{CL}$ has to pass through $n_3-n_3'$ nodes of $\mathcal{A}$, and hence $n_2 \geq n_2'-(n_3-n_3')$. It follows that
$$n_2+n_3 \geq n_2'+n_3' =21-2n_3'.$$
Hence $n_3' \geq 6$. Using the classification of line arrangements with ${\rm mdr}(f) \leq 2$ presented in \cite{BT}, it follows that our arrangement $\mathcal{A}$ satisfies ${\rm mdr}(f) \geq 3$. A calculation of the total Tjurina number, under the assumption that $n_3' \geq 6$, shows us that $n_3' = 6$, $n_2' = 3$, ${\rm mdr}(f)=3$, and the arrangement $\mathcal{A}$ has to be free.
It follows that the arrangement $\mathcal{A}$ is projectively equivalent to
$$\A_0: xyz(x+y)(x+z)(y-z)(x+y+z)=0,$$
see \cite[Theorem 2.6]{BT}, where this arrangement occurs as \textbf{IIIc}. Moreover, exactly as in the previous proof, the intersections between the conic and the lines in $\mathcal{A}$ should not add any new point to the set of seven multiple points of $\A$. Note that $4$ of the $7$ lines, denoted here by $L_2,L_3,L_4$, and $L_5$, where $L_j$ means the line given by the $j$-th factor in the equation of $\mathcal{A}$, contain each $3$ triple points, while the remaining $3$ lines, $L_1, L_6$ and $L_7$, contain each $2$ triple points and $2$ double points. 
This implies that the conic is tangent to $4$ lines $L_2,L_3,L_4$, and $L_5$ (in new points given rise to tacnodes), and passes through the remaining $3$ double points, located at 
$$(0:1:1), \ (0:-1:1) \text{ and } (-2:1:1). $$
A direct computation shows that such a conic does not exist.
\end{proof}

In conclusion, we can state the following complete classification result.
\begin{thm}
\label{thm11}
Let $\mathcal{CL}$ be an arrangement of $d \geq 1$ lines and $k \geq 1$ smooth conics having only nodes, tacnodes, and  ordinary triple points as singularities. Then  $\mathcal{CL}$ is free if and only if one of the following cases occur:
In each case we list the numbers $n_2$, $t$, and $n_3$ of nodes, tacnodes, and ordinary triple points, respectively.
\begin{enumerate}

\item $d=k=1$ and $\mathcal{CL}$ consists of a smooth conic and a tangent line.
In this case, $n_2=n_3=0$, $t=1$.

\item $d=2$, $k=1$ and $\mathcal{CL}$ consists of a smooth conic and two tangent lines. In this case $n_2=1$, $n_3=0$, $t=2$.

\item $d=3$, $k=1$ and either $\mathcal{CL}$ is a smooth conic inscribed in a triangle, or $\mathcal{CL}$ is a smooth conic circumscribed in a triangle.
In the first case we have $n_2=3$, $n_3=0$, $t=3$, and in the second case we have $n_2=t=0$, $n_3=3$. 

\item $d=3$, $k=2$ and $\mathcal{CL}$ consists of a triangle $\Delta$, a smooth conic inscribed in $\Delta$, and another smooth conic circumscribed in $\Delta$. In this case, $n_2=0$, $n_3=3$, $t=5$.

\end{enumerate}
In particular, a free conic-line arrangement having only nodes, tacnodes, and ordinary triple points is determined up to a projective equivalence by the numerical data $n_2$, $n_3$ and $t$.
\end{thm}
Before we formulate the final corollary for this section, we need the following notations inspired by \cite{Mar}.
\begin{definition}
We say that two conic-line arrangements in $\mathbb{P}^{2}_{\mathbb{C}}$ with nodes, tacnodes, and ordinary triple points have the same \textbf{weak combinatorics} if these arrangements have the same list of invariants $(m; n_{2}, t, n_{3})$ with $m$ being degree of the arrangements.
\end{definition}
\begin{conj}[Numerical Terao's Conjecture]
Let $\mathcal{CL}_{1}, \mathcal{CL}_{2} \subset \mathbb{P}^{2}_{\mathbb{C}}$ be two conic-line arrangements with nodes, tacnodes, and ordinary triple points. Assume that $\mathcal{CL}_{1}$ is free and  $\mathcal{CL}_{1}$, $\mathcal{CL}_{2}$ have the same weak combinatorics, then $\mathcal{CL}_{2}$ is also free.
\end{conj}

\begin{cor}
\label{cor5}
Numerical Terao's Conjecture holds for conic-line arrangements with nodes, tacnodes, and ordinary triple points.
\end{cor}
\proof
Note that the equation \eqref{E4} implies that the list of invariants $(m; n_{2}, t, n_{3})$ determines the number $k$ of conics  and the number $d=m-2k$ of lines in the arrangements  $\mathcal{CL}_{1}$ and  $\mathcal{CL}_{2}$. 
Then Theorem \ref{thm11} implies that $k=1$ or $k=2$.
In each case, using the fact that $k$ and $d$ are very small, it is easy to see that  up to a projective transformation, the possibilities for $\mathcal{CL}_{2}$ are exactly those listed in Theorem \ref{thm11}, and hence the arrangement $\mathcal{CL}_{2}$ is free as well.
\endproof
\begin{rk}
Numerical Terao's Conjecture can be formulated, in principle, for all reduced singular plane curves. As it was showed in \cite{Mar}, Numerical Terao's Conjecture fails for some (triangular) line arrangements. 
On the other hand, it holds for line arrangements having only points of multiplicity $\leq 3$. Indeed, Proposition \ref{prop:bound9} shows that such a free line arrangement $\mathcal{A}: f=0$ has to satisfy $m= \deg f \leq 9$.
Then Theorem \ref{thmF} implies that either $d_1= {\rm mdr}(f) \leq 3$ or
$d_1= {\rm mdr}(f) = 4$ and $d=9$. Note that if $\mathcal{A}':f'=0$ has the same weak combinatorics as $\A:f=0$, then $\tau(\A)=\tau(\A')$, which implies that
$r'= {\rm mdr}(f') \leq {\rm mdr}(f)$ using the maximality of the Tjurina number for free reduced curves, see \cite{duPlessisWall}.
In the first case, one concludes using the complete classification of line arrangements with ${\rm mdr}(f) \leq 3$, see \cite{BT}. In the second case, we use again the maximality of the Tjurina number of free curves according to \cite{duPlessisWall} and we conclude that $\tau(\A)=\tau(\A')=48$, $n_2=0$ and $n_3=12$. The only line arrangement with these invariants is the line arrangement in Example \ref{exH1} above, which is indeed free with the exponents $(4,4)$.
\end{rk}

\section*{Funding} The first author was partially supported by the Romanian Ministry of Research and Innovation, CNCS - UEFISCDI, Grant \textbf{PN-III-P4-ID-PCE-2020-0029}, within PNCDI III. The second author was partially supported by the National Science Center (Poland) Sonata Grant Nr \textbf{2018/31/D/ST1/00177}. We want to thank an anonymous referee for comments on that paper.
\section*{Data availability} Not applicable as the results presented in this manuscript rely on no external sources of data or code.

\section{Addendum}
In this section we want to update our paper by showing that there exists another geometric realization of the weak-combinatorics in Theorem \ref{thm11}(4).

Consider the conic--line arrangement
\[
\mathcal{CL}_{\mathrm{con}}:\quad
f=xy(x-y)(z^2+xy)\big((y-x-z)^2+xy\big)=0,
\]
consisting of the three concurrent lines
\[
L_1:\ x=0,\qquad L_2:\ y=0,\qquad L_3:\ x-y=0
\]
and the two smooth conics
\[
Q_1:\ z^2+xy=0,\qquad
Q_2:\ (y-x-z)^2+xy=0.
\]
The arrangement has weak combinatorics
\[
(d,k;n_2,t,n_3)=(3,2;0,5,3).
\]
Indeed, \(Q_1\) is tangent to \(L_1\) and \(L_2\) at
\[
(0:1:0)\quad\text{and}\quad (1:0:0),
\]
respectively, whereas \(Q_2\) is tangent to \(L_1\) and \(L_2\) at
\[
(0:1:1)\quad\text{and}\quad (1:0:-1).
\]
Moreover,
\[
Q_1-Q_2=(y-x)(2z-y+x),
\]
and restricting to the line \(2z=y-x\) gives
\[
Q_1=\frac{(x+y)^2}{4},
\]
so \(Q_1\) and \(Q_2\) have intersection multiplicity \(2\) at
\[
(1:-1:-1),
\]
which yields the fifth tacnode. The three ordinary triple points are
\[
(0:0:1),\qquad (1:1:i),\qquad (1:1:-i).
\]
The first is the common point of \(L_1,L_2,L_3\), while at the latter
two points the three components \(Q_1,Q_2,L_3\) have pairwise distinct
tangent directions. Hence all five tangential double intersections are
of type \(A_3\), all three triple points are ordinary \(D_4\)-points,
and there are no nodes.  Finally, this strong realization
type is projectively rigid: after normalizing the three concurrent
lines to \(x=0\), \(y=0\), and \(x-y=0\), every pair of smooth conics
with the same incidence and tangency data is equivalent, under the
residual subgroup of \(\operatorname{PGL}_3(\mathbb{C})\) preserving
these three lines, to the pair \(Q_1,Q_2\) displayed above. Moreover, we can check that 
$\mathcal{CL}_{\mathrm{con}}$ is free with exponents $(3,3)$.

This means that the weak combinatorics $(d,k;n_2,t,n_3)=(3,2;0,5,3)$ does not uniquely determine the arrangement of conics and lines, and we have two branches in Theorem \ref{thm11}(4). 
\end{document}